\documentclass[12pt,reqno]{amsart}

\usepackage[latin1]{inputenc}
\usepackage{amsmath}
\usepackage{amsfonts}
\usepackage{amssymb,esint}
\usepackage{mathrsfs}
\usepackage{hyperref}

\theoremstyle{definition}
\newtheorem{definition}{Definition}[section]
\theoremstyle{plain}
\newtheorem{lemma}[definition]{Lemma}
\newtheorem{theorem}[definition]{Theorem}

\newtheorem{rmk}[definition]{Remark}

\newcommand{\el}{\ell}

\newcommand{\R}{\mathbb{R}}
\newcommand{\N}{\mathbb{N}}
\renewcommand{\subset}{\subseteq}
\newcommand{\G}{\mathbb{G}}

\author[Don]{Sebastiano Don}
\author[Vittone]{Davide Vittone}
\address[Don and Vittone]{Dipartimento di Matematica, via Trieste 63, 35121 Padova, Italy.}
\email{sebastiano.don@math.unipd.it,vittone@math.unipd.it}

\thanks{The authors are supported by the University of Padova, through Project Networking and STARS Project ``Sub-Riemannian Geometry and Geometric Measure Theory Issues: Old and New'' (SUGGESTION), and by GNAMPA of INdAM (Italy), through project ``Campi vettoriali, superfici e perimetri in geometrie singolari''. }



\begin{document}

\title{A compactness result for BV functions in metric spaces}

\begin{abstract}
We prove a compactness result for bounded sequences $(u_j)_j$ of functions with bounded variation in metric spaces $(X,d_j)$ where the space $X$ is fixed but the metric may vary with $j$. We also provide an application to  Carnot-Carath\'eodory spaces.
\end{abstract}

\maketitle\vspace{-1mm}

\section{Introduction}
One of the milestones in the theory of functions with bounded variation (BV) is the following Rellich-Kondrachov-type theorem: given a bounded open set $\Omega\subset\R^n$ with Lipschitz regular boundary, the  space $BV(\Omega)$ of functions with bounded variation in $\Omega$ compactly embeds in $L^q(\Omega)$ for any $q\in[1,\frac{n}{n-1}[$. One notable consequence  is the following property: if $(u_j)_j$ is a  sequence of functions in $BV_{loc}(\R^n)$ that are locally uniformly bounded in $BV$, then for any $q\in[1,\frac{n}{n-1}[$ a subsequence $(u_{j_h})_h$ converges in  $L^q_{loc}(\R^n)$.

Sobolev and BV functions in metric measure spaces have recently received a great deal of attention; to this regard we only mention the celebrated paper \cite{hajkos}, where the authors show how the validity of Poincar\'e-type inequalities and a doubling property of the reference measure are enough to prove fundamental properties  like Sobolev inequalities, Sobolev embeddings, Trudinger inequality, etc. We also point out a Rellich-Kondrachov-type result \cite[Theorem 8.1]{hajkos}: if a sequence $(u_j)_j$ is bounded in some $W^{1,p}$, then a subsequence converges in some $L^q$. 

In this paper we study similar compactness properties for sequences $(u_j)_j$ of locally uniformly bounded BV functions in  metric measure spaces $(X,\lambda,d_j)$ where the underlying measure space $(X,\lambda)$ is fixed but the metric $d_j$ varies with $j$. In our main result we prove that, if $d_j$ converges locally uniformly to some distance $d$ on $X$ such that $(X,\lambda,d)$ is a (locally) doubling separable metric measure space, and if the functions $u_j:X\to \R$ are locally uniformly (in $j$) bounded with respect to a BV-type norm in $(X,d_j)$ and satisfy some local Poincar\'e inequality (with constant independent of $j$), then a subsequence of $u_j$ converges in some $L^q_{loc}(X,\lambda)$. See Theorem \ref{structureconvergence} for a precise statement. To our knowledge, the strategy we adopt to prove Theorem \ref{structureconvergence} is novel even when the metric on $X$ is not varying  (i.e., when $d_j=d$ for any $j$); in particular,  we are able to provide a different proof of the case $p=1$ in \cite[Theorem 8.1]{hajkos} for separable metric spaces. 

The motivation that led us to Theorem \ref{structureconvergence} comes from an application   to the study of $BV$ functions in {\em Carnot-Carath\'eodory} (CC) spaces. In Theorem \ref{teo:applicazione} we indeed prove that, if $X^j=(X_1^j,\dots,X^j_m)$ are families of smooth vector fields in $\R^n$ that, as $j\to\infty$, converge in $C^\infty_{loc}(\R^n)$ to a family $X=(X_1,\dots,X_m)$ satisfying the Chow-H\"ormander condition, and if $u_j:\R^n\to\R$ are locally uniformly bounded in $BV_{X^j,loc}$, then a subsequence $u_{j_h}$ converges in $L^1_{loc}(\R^n)$ to some $u\in BV_{X,loc}(\R^n)$.  Theorem \ref{teo:applicazione} directly follows from Theorem \ref{structureconvergence} once we show that the CC distances induced by $X^j$ converge locally uniformly to the one induced by $X$, and that (locally) a Poincar\'e inequality holds for $BV_{X^j}$ functions with constant independent of $j$; these two results (Theorems \ref{teo:convunifd_j} and \ref{teo:poincare}, respectively) use in a crucial way some outcomes of the papers \cite{BBP,MM}.

Our interest in Theorem \ref{teo:applicazione}, in turn,  was originally motivated by  the study  of fine properties of $BV_X$ functions in CC spaces and, in particular, of their local properties.  Here, one often needs to perform a blow-up procedure around a fixed point $p$: it is well-known that this produces a sequence of CC metric spaces $(\R^n,X^j)$ that converges to (a quotient of) a {\em Carnot group}  structure $\G$. In this blow-up, the original $BV_X$ function $u_0$ gives rise to a sequence $(u_j)_j$ of functions in $BV_{X^j}$ which, up to a subsequence, will converge in $L^1_{loc}$ to a $BV_{\G,loc}$ function $u$ in $\G$. The function $ u$ (typically: a linear map, or a {\em jump map} taking two different values on complementary halfspaces of $\G$) will then provide some information on $u_0$ around $p$. We refer to \cite{DVfineproperties} for more details.

{\em Aknowledgements.} The authors are grateful to M.~Miranda Jr. and D.~Morbidelli for fruitful discussions.

\section{The main result}
This section is devoted to the statement and the proof of our main result. See e.g. \cite{MirandaBVgood} for a definition of BV functions in metric spaces.

\begin{theorem}\label{structureconvergence}
	Let $X$ be a set, $q\geq 1$, $\delta>0$ and let $d, d_j$ ($j\in \mathbb{N}$) be metrics on $X$ such that $(X,d)$ is locally compact and separable. Let $\lambda,\, \mu_j \,(j\in \mathbb{N})$ be Radon measures on $X$ and consider a sequence $(u_j)_j$ in $L^q_{loc}(X;\lambda)$. Suppose that the following assumptions hold.
\begin{itemize}
	\item[(i)] The sequence $(d_j)_j$ converges to $d$ in $L^\infty_{loc}(X\times X)$.
	\item[(ii)] $(X,d,\lambda)$ is a locally doubling metric measure space, i.e., for any compact set $K\subseteq X$ there exist $C_D\geq 1$ and $R_D>0$ such that
	\[
	\forall\ x\in K,\ \forall r\in (0,R_D) \qquad \lambda(B(x,2r))\leq C_D\lambda (B(x,r)).
	\]
	\item[(iii)] For every compact set $K\subseteq X$ there exist $C_P, R_P>0$ and $\alpha\geq 1$ such that 
	\[
	\forall x\in K,\ \forall j\in \mathbb{N},\ \forall r\in (0,R_P)\qquad \|u_j-u_j(B^j)\|_{L^q(B^j)}\leq C_P\,r^\delta\mu_j(\alpha B^j),
	\]
where $B^j:=B^j(x,r)$ denotes a ball in $(X,d_j)$, $\alpha B^j:=B^j(x,\alpha r)$ and $ u_j(B^j):=\fint_{B^j} u_j d\lambda$.
	\item[(iv)] For every compact set $K\subseteq X$ there exists $M_K>0$ such that
	\[
	\forall j\in \mathbb{N}\qquad \|u_j\|_{L^1(K;\lambda)}+\mu_j(K)\leq M_K.
	\]
\end{itemize}
	Then there exist $u\in L^q_{loc}(X;\lambda)$ and a subsequence $(u_{j_h})_h$ of $(u_j)_j$ such that $(u_{j_h})_h$ converges to $u$ in $L^q_{loc}(X;\lambda)$ as $h\to +\infty$. 
\end{theorem}

Concerning the classical Euclidean case when $(X,d_j,\lambda)=(X,d,\lambda)=(\R^n,|\cdot|,\mathscr L^n)$, we invite the reader to compare the assumption in (iii) with the well-known Poincar\'e inequality
\[
\|u-u(B_r)\|_{L^q(B_r)} \leq C r^\delta |Du|(B_r)\qquad \forall\;q\in[1,\tfrac n{n-1}[\text{ with } \delta:=\tfrac nq +1-n>0 
\]
valid for  for any $BV$ function $u$ on any ball $B_r\subset\R^n$ of radius $r$ and where $u(B_r)$ denotes the mean value $\mathscr L^n(B_r)^{-1}\int_{B_r}u\,d\mathscr L^n$ of $u$ in $B_r$, $C>0$ is a geometric constant, and $|Du|$ denotes the total variation measure associated with $u$ (i.e., the total variation of the distributional derivatives of $u$).

\begin{proof}
	Let $K\subseteq X$ be a fixed compact set and let $\varepsilon>0$. We first prove that there exists a subsequence $(u_{j_h})_h$ such that
\begin{equation}\label{eq:primoclaim}
	\limsup_{h,k \to +\infty}\|u_{j_h}-u_{j_k}\|_{L^q(K;\lambda)}\leq 2 C_0 \varepsilon,
\end{equation}
for some $C_0>0$ depending on $K$ only.

	Consider an open set $U_1\subseteq X$ such that $K\subseteq U_1$, $\overline{U}_1$ is compact and 
\begin{equation}\label{eq1}
	\lambda(U_1\setminus K)\leq \frac{1}{4C_D^{\beta+3}}\lambda(K).
\end{equation}
where $\beta$ is an integer such that $2^\beta>2\alpha$ and $\alpha$ is given by condition (iii).
	By the $5r-$covering Theorem (see e.g. \cite[Theorem 1.2]{Heinonen}) we can find a family $\displaystyle\left\{B(x_\el,r_\el): \el\in \mathbb{N}\right\}$ of pairwise disjoint balls such that $x_\el\in K$, $0<r_\el<\min\{\varepsilon^{1/\delta}, R_D/4, 2\alpha R_P\}$, $\overline{B(x_\el,5r_\el)}\subseteq U_1$ and 
\[
	K\subseteq \bigcup_{\el=0}^\infty \overline{B(x_\el,5r_\el)}.
\]
	Denote for shortness $B_\el:=B(x_\el,r_\el)$; then
\[
\lambda(K)\leq \sum_{\el=0}^\infty \lambda(5\overline{B}_\el)\leq \sum_{\el=0}^\infty \lambda(8B_\el)\leq C_D^{\beta+3}\sum_{\el=0}^\infty \lambda(\tfrac{1}{2^\beta}B_\el)= C_D^{\beta+3}\lambda\left(\bigcup_{\el=0}^\infty \tfrac{1}{2^\beta}B_\el\right).
\]
	Hence we can choose $L\in \mathbb{N}$ such that
\[
	\lambda\left(\bigcup_{\el=0}^L\tfrac{1}{2^\beta} {B}_\el\right)\geq \frac{1}{2C_D^{\beta+3}}\lambda(K).
\]
	Taking into account \eqref{eq1} we easily get that $A_1:=K\cap \bigcup_{\el=0}^L\tfrac{1}{2^\beta}B_\el$ satisfies
\[
	\lambda(A_1)\geq\frac{1}{4C_D^{\beta+3}}\lambda(K).
\]
	For $j\in \mathbb N$ and $\el=1,\dots, L$ set for shortness $B_\el^j:=B^j(x_\el,r_\el)$. By assumption (i) there exists $J\in \mathbb{N}$ such that for every $j\geq J$, and for every $\el=0,\dots, L$
\begin{equation}\label{eq:contenimenti}
  \tfrac{1}{2^\beta}B_\el\subseteq \tfrac{1}{2\alpha}B_\el^j\qquad\text{and}\qquad
	\tfrac{1}{2}B_\el^j\subseteq B_\el.
\end{equation}
Hence for every $j\geq J$ one has
\[
\left|u_j\left(\tfrac{1}{2\alpha}B_\el^j\right)\right|
 \leq \lambda\left(\tfrac{1}{2\alpha}B_\el^j\right)^{-1}\|u_j\|_{L^1(U_1;\lambda)}\\
\leq M_{\overline{U_1}}\max\{\lambda\left(\tfrac{1}{2^\beta}B_\el\right)^{-1}:\el=0,\dots,L\} <+\infty.
\]
	By Bolzano-Weierstrass Theorem we get an increasing function $\nu_1:\mathbb{N}\rightarrow\mathbb{N}$ such that
\begin{equation}\label{eq:BWT}
\text{the sequence $\left(u_{\nu_1(j)}\left(\tfrac{1}{2\alpha}B^{\nu_1(j)}_\el\right)\right)_j$ is convergent for every $\el=0,\dots,L$.}
\end{equation}
Then
\begin{align}
	&\limsup_{h,k\to+\infty}\|u_{\nu_1(h)}-u_{\nu_1(k)}\|_{L^q(A_1;\lambda)}\nonumber\\
	\leq& \limsup_{h,k\to+\infty}\sum_{\ell=0}^L\left(\left\|u_{\nu_1(h)}-u_{\nu_1(h)}\left(\tfrac{1}{2\alpha}B_\el^{\nu_1(h)}\right)\right\|_{L^q\left(\tfrac{1}{2^\beta}B_\el;\lambda\right)}\right.\nonumber\\
	&\hphantom{\leq\limsup_{h,k\to+\infty}\sum}+\left\|u_{\nu_1(k)}-u_{\nu_1(k)}\left(\tfrac{1}{2\alpha}B_\el^{\nu_1(k)}\right)\right\|_{L^q\left(\tfrac{1}{2^\beta}B_\el;\lambda\right)}\nonumber\\
	&\hphantom{\leq\limsup_{h,k\to+\infty}\sum}\left.+\left\|u_{\nu_1(h)}\left(\tfrac{1}{2\alpha}B_\el^{\nu_1(h)}\right)-u_{\nu_1(k)}\left(\tfrac{1}{2\alpha}B_\el^{\nu_1(k)}\right)\right\|_{L^q\left(\tfrac{1}{2^\beta}B_\el;\lambda\right)}\right)\nonumber\\
	\intertext{and, using \eqref{eq:contenimenti} and \eqref{eq:BWT},}
	\leq &\limsup_{h,k\to +\infty}\sum_{\el=0}^L\left(\left\|u_{\nu_1(h)}-u_{\nu_1(h)}\left(\tfrac{1}{2\alpha}B_\el^{\nu_1(h)}\right)\right\|_{L^q\left(\tfrac{1}{2\alpha}B_\el^{\nu_1(h)};\lambda\right)}\right.\nonumber\\
	&\hphantom{\leq\limsup_{h,k\to+\infty}\sum}\left.+\left\|u_{\nu_1(k)}-u_{\nu_1(k)}\left(\tfrac{1}{2\alpha}B_\el^{\nu_1(k)}\right)\right\|_{L^q\left(\tfrac{1}{2\alpha}B_\el^{\nu_1(k)};\lambda\right)}\right)\nonumber\\
	\leq&\limsup_{h,k\to+\infty} \sum_{\el=0}^L\frac{C_P\,r_\el^{\delta}}{(2\alpha)^{\delta}}\left(\mu_{\nu_1(h)}\left(\tfrac{1}{2}B^{\nu_1(h)}_\el\right)+\mu_{\nu_1(k)}\left(\tfrac{1}{2}B^{\nu_1(k)}_\el\right)\right)\nonumber\\
	\leq&\limsup_{h,k\to+\infty}\frac{C_P\,\varepsilon}{(2\alpha)^{\delta}}\left(\mu_{\nu_1(h)}\left(\overline{U}_1\right)+\mu_{\nu_1(k)}\left(\overline{U}_1\right)\right)\leq C_0\varepsilon,\nonumber
	\end{align}
where $C_0$ depends only on $U_1$ and thus only on $K$.\newline
	We  proved that there exist $A_1\subseteq K$ and a subsequence $(u_{\nu_1(h)})_h$ of $(u_j)_j$ such that
\[
\begin{aligned}
	&\lambda(K\setminus A_1)\leq \left(1-\frac{1}{4C_D^{\beta+3}}\right)\lambda(K),\\
	&\limsup_{h,k\to+\infty}\|u_{\nu(h)}-u_{\nu(k)}\|_{L^q\left(A_1;\lambda\right)}\leq C_0 \varepsilon.
\end{aligned}
\]
	Since the set $K_2=K\setminus A_1$ is compact we can repeat the same argument  on $K_2$, with $\tfrac{\varepsilon}{2}$ in place of $\varepsilon$, and paying attention to choose an open set $U_2\subseteq U_1$ so that $C_0$ can be left unchanged. By a recursive argument, for every $j\in \mathbb{N}$ we get pairwise disjoint sets $A_j\subseteq K$ and subsequences $(u_{\nu_j(h)})_h$  such that for every $j\geq 1$
\begin{itemize}
	\item[(a)] $(u_{\nu_{j+1}(h)})_h$ is a subsequence of $(u_{\nu_j(h)})_h$;
	\item[(b)] 
	$\lambda\left(K\setminus \bigcup_{i=1}^jA_i\right)\leq \left(1-\frac{1}{4C_D^{\beta+3}}\right)^j\lambda(K)$;
	\item[(c)] $\displaystyle \limsup_{h,k\to+\infty} \|u_{\nu_j(h)}-u_{\nu_j(k)}\|_{L^q\left(A_j;\lambda\right)}\leq C_02^{1-j}\varepsilon$.
\end{itemize}
	Inequality (b) immediately implies that $\lambda\left(K\setminus\bigcup_{i=1}^\infty A_i\right)=0$. Working on the diagonal subsequence $(u_{\nu_h(h)})_h$ we can conclude that
\begin{equation}\label{limsup}
	\begin{aligned}
	\limsup_{h,k\to+\infty}\|u_{\nu_h(h)}-u_{\nu_k(k)}\|_{L^q\left(K;\lambda\right)}&=\limsup_{h,k\to+\infty}\|u_{\nu_h(h)}-u_{\nu_k(k)}\|_{L^q\left(\bigcup_{i=1}^\infty A_i;\lambda\right)}\\
	&\leq \sum_{i=1}^\infty \limsup_{h,k\to+\infty}\|u_{\nu_h(h)}-u_{\nu_k(k)}\|_{L^q\left(A_i;\lambda\right)}\leq 2 C_0\varepsilon.
\end{aligned}
\end{equation}
This proves \eqref{eq:primoclaim}.

	Let us denote for simplicity $(u_h)_h$ instead of $(u_{\nu_h(h)})_h$. We now prove that for every compact set $K\subseteq X$ there exists a subsequence $(u_{j_h})_h$ of $(u_h)_h$ such that 
\begin{equation}\label{diagonalcompact}
	\lim_{h,k\to+\infty}\|u_{j_h}-u_{j_k}\|_{L^q\left(K;\lambda\right)}=0.
\end{equation}
	By \eqref{limsup}, for every $i\in \mathbb N$, we can  recursively build a subsequence $(u_{\nu_{i+1}(h)})_h$ of $(u_{\nu_i(h)})_h$ such that
	\[
	\limsup_{h,k \to +\infty}\|u_{\nu_i(h)}-u_{\nu_i(k)}\|_{L^q(K;\lambda)}\leq \tfrac{2}{i+1}C_0.
	\]
	Then the diagonal sequence $(u_{\nu_h(h)})$ satisfies \eqref{diagonalcompact}.
	
Eventually, take a sequence $(K_j)$ of compact sets such that $K_j\subseteq \mathrm{int}(K_{j+1})$ and $\bigcup_{j\in \mathbb N}K_j=X$. By \eqref{diagonalcompact}, for every $i\in \mathbb{N}$ we can recursively build a subsequence $(u_{\nu_i(h)})_h$  such that $(u_{\nu_{i+1}(h)})_h$ is a subsequence of $(u_{\nu_i(h)})_h$ and 
	\[
	\lim_{h,k\to+\infty}\|u_{\nu_i(h)}-u_{\nu_i(k)}\|_{L^q\left(K_i;\lambda\right)}=0.
	\]
The diagonal subsequence $(u_{\nu_h(h)})_h$ will then  converge to some $u$ in $L^q_{loc}(X;\lambda)$. This concludes the proof.
\end{proof}

\begin{rmk}{\rm 
The careful reader will easily notice that Theorem \ref{structureconvergence} holds also when assumption (iii) is replaced by the following weaker one:
\begin{itemize}
	\item[(iii')] For every compact set $K\subseteq X$ there exist $R_P>0,\alpha\geq 1$ and  $f:(0,+\infty)\:\to(0,+\infty)$ such that  $\lim_{r\to 0^+}f(r)=0$ and
	\[
	\forall x\in K,\ \forall j\in \mathbb{N},\ \forall r\in (0,R_P)\qquad \|u_j-u_j(B^j)\|_{L^q(B^j)}\leq f(r)\,\mu_j(\alpha B^j).
	\]
\end{itemize}
}\end{rmk}

\section{An application to Carnot-Carath\'eodory spaces}
	Let $\Omega$ be an open set in $\mathbb{R}^n$ and let $X=(X_1,\dots,X_m)$ be an $m$-tuple  of smooth and  linearly independent vector fields on $\mathbb{R}^n$, with $2\leq m\leq n$. We say that an absolutely continuous curve $\gamma:[0,T]\rightarrow\mathbb{R}^n$ (briefly denoted by $\gamma\in AC([0,T];\mathbb{R}^n)$) is an $X$-subunit path joining $x$ and $y$ in $\mathbb{R}^n$ if $\gamma(0)=x$, $\gamma(T)=y$ and there exist $h_1,\dots,h_m:[0,T]\rightarrow\mathbb R$ with $\sum_{j=1}^m h_j^2\leq 1$ such that
	\begin{equation}\label{eq:orizzontalecontrollih}
		 \dot{\gamma}(t)=\sum_{j=1}^m h_j(t)X_j(\gamma(t))\qquad\text{for a.e. }t\in[0,T].
	\end{equation}
	Moreover, for every $x,y\in \mathbb R^n$ we define the quantity
\begin{equation}\label{eq:defdcc}
	d(x,y):=\inf\big\{T\in(0,+\infty): \exists \gamma\in AC([0,T];\mathbb{R}^n) \text{ $X$-subunit joining } x \text{ and } y\big\},
\end{equation}
where we agree that $\inf\emptyset=+\infty$. 

We will suppose in the following that the Chow-H\"ormander condition holds, i.e., that for every $x\in \mathbb{R}^n$  the vector space spanned by $X_1,\dots,X_m$ and their commutators of any order computed at $x$ is the whole $\mathbb{R}^n$. By the Chow-Rashevsky Theorem, if the Chow-H\"ormander condition holds, the function $d$ defined above is  a distance and the couple $(\mathbb R^n, X)$ (or equivalently $(\mathbb R^n, d)$) is called {\em Carnot-Carath\'eodory space} (CC space for short). 	It is well known that $d$ and the Euclidean distance $d_e$ induce on $\mathbb R^n$ the same topology (see \cite{NSW}).
	
We  denote  balls induced by $d$ by $B(x,r)$ and  Euclidean balls by $B_e(x,r)$. As customary in the literature, in what follows we also suppose that the metric balls $B(x,r)$ are bounded with respect to the Euclidean metric. One consequence of this assumption is the existence of geodesics, i.e., for any $x,y\in\R^n$ the infimum in \eqref{eq:defdcc} (as well as the one in \eqref{eq:defdcc2} below) is indeed a minimum; see e.g. \cite[Theorem 1.4.4]{Monti}.

For $j\in\mathbb N$ let $X^j=(X_1^j,\dots,X_m^j)$ be a family of linearly independent vector fields such that, for every fixed $i=1,\dots, m$,  $X_i^j$ converges to $X_i$ in $C^\infty_{loc}(\mathbb R^n)$ as $j\to\infty$. We denote by $d_j$, $j\in \mathbb N$, the CC distance associated with $X^j$. 	If $h\in L^\infty([0,T];\mathbb{R}^m)$ with $\|h\|\leq 1$, $T>0$ and $x\in \mathbb{R}^n$, it is convenient to define $\gamma_{h,x},\gamma^j_{h,x}:[0,T]\rightarrow\mathbb{R}^n$ as the AC curves such that $\gamma_{h,x}(0)=\gamma^j_{h,x}(0)=x$ and for almost every $t\in [0,T]$
	\[
	\dot\gamma_{h,x}(t)=\sum_{i=1}^m h_i(t)X_i(\gamma_{h,x}(t)),\qquad \dot\gamma^j_{h,x}(t)=\sum_{i=1}^m h_i(t)X_i^j(\gamma^j_{h,x}(t)).
	\]
With this notation, an equivalent definition of the CC distance is
\begin{equation}\label{eq:defdcc2}
d(x,y)=\inf\{\|h\|_{L^\infty(0,1)}:h\in L^\infty([0,1];\mathbb{R}^m)\text{ and }\gamma_{h,x}(1)=y\}.
\end{equation}
The boundedness of metric balls implies that,  for every $T>0$ and $h\in L^\infty([0,T];\R^m)$, the curve $\gamma_{h,x}$ is well-defined on $[0,T]$. 
	
It can be easily seen that, if the Chow-H\"ormander condition holds, then for every compact set $K\subset\R^n$ there exists an integer $s(K)$ such that the following holds: for any $x\in K$, $X_1,\dots,X_m$ and their commutators up to order $s(K)$ computed at $x$ span the whole $\R^n$. The following  theorem gives a sort of quantitative version of some of the celebrated results of \cite{NSW}. The proof of Theorem \ref{uniformity} follows fairly easily  from  \cite{BBP,MM} (see in particular \cite[Proposition 5.8 and Claim 3.3]{BBP}) and from the following observation: for any compact set $K\subset\R^n$ there exists $J\in\N$ such that, for any $x\in K$ and $j\geq J$, the vector fields $X_1^j,\dots,X_m^j$ and their commutators up to order $s(K)$ computed at $x$ span the whole $\R^n$.
	
\begin{theorem}\label{uniformity}
For every compact set $K\subseteq \mathbb R^n$ there exist $J_0\in\N$ and $C_K>0$ such that for every $x,y\in K$  and $j\geq J_0$
\[
\begin{split}
&\frac1{C_K}|x-y|\leq d(x,y)\leq C_K|x-y|^{1/{s(K)}}\\
&\frac1{C_K}|x-y|\leq d_j(x,y)\leq C_K|x-y|^{1/s(K)}.
\end{split}
\]
\end{theorem}

We aim at proving that the sequence of distances $d_j$ converges to $d$ locally uniformly; we need some preparatory lemmata.
	
\begin{lemma}\label{welldefined}
	Let $K$ be a compact set in $\R^n$. Then for every $T>0$, there exist $J_1=J_1(K,T)\in \mathbb N$ and $R=R(K,T)>0$  such that for every $x\in K$, $h\in L^\infty([0,T];\R^m)$ with $\|h\|\leq 1$ and any $j\geq J_1$ the following hold:
	\begin{itemize}
	\item[(a)] the curve $\gamma_{h,x}^j$ is well defined on $[0,T]$;
	\item[(b)] $\gamma_{h,x}^j([0,T])\subseteq B_e(0,R)$.
	\end{itemize}
\end{lemma}
\begin{proof}
	Define first 
	\[
	K':=\{\gamma_{h,x}(T):x\in K, h\in L^\infty([0,T];\R^m), \|h\|\leq 1\}=\bigcup_{x\in K} \overline{B(x,T)}.
	\]
	Since metric balls are bounded, also $K'$ is bounded. We can therefore find $R>0$ such that $K'\subseteq B_e(0,R)$ and $d_e(K',\R^n\setminus B_e(0,R))>1$. Choose $J_1\in \mathbb N$  such that for every $j\geq J_1$
	\[
	T\left(\sum_{i=1}^m\sup_{B_e(0,R)}|X_i^j-X_i|\right)e^{mCT}\leq \frac{1}{2},
	\]
	where $C>0$ will be determined later. Let $h\in L^\infty([0,T];\R^m)$ and $j\geq J_1$ be fixed; define
	\[
	t_j:=\sup\{t>0: \gamma_{h,x}^j \text{ is well-defined on } [0,t] \text{ and } \gamma_{h,x}^j([0,t])\subseteq B_e(0,R)\}
	\]
	and suppose by contradiction that $t_j<T$. Then $\gamma_{h,x}^j(t_j) \in \partial B_e(0,R)$ and for every $\tau<t_j$ one has
\[
	\begin{aligned}
	\left|\gamma^j_{h,x}(\tau)-\gamma_{h,x}(\tau)\right|&\leq \int_0^\tau\sum_{i=1}^m\left|h_i(s)X^j_i(\gamma^j_{h,x}(s))- h_i(s)X_i(\gamma_{h,x}(s)) 	\right|ds\\
	&\leq\int_0^\tau\sum_{i=1}^m\left|X_i^j(\gamma^j_{h,x}(s))-X_i^j(\gamma_{h,x}(s))\right|ds \\&\hphantom{\leq}+\int_0^\tau\sum_{i=1}^m\left|X_i^j(\gamma_{h,x}(s))-X_i(\gamma_{h,x}(s))\right|ds.
	\end{aligned}
\]
	Notice that, since $X^j_i$ is converging to $X_i$ locally in $C^1$, and since $\gamma^j_{h,x}(s), \gamma_{h,x}(s)\in B_e(0,R)$, the Lipschitz constants
	\[
	c_i^{\,j}:=\sup_{x,y \in B_e(0,R)}\frac{|X_i^j(x)-X_i^j(y)|}{|x-y|}
	\]
are converging to the Lipschitz constant  $ c_i:=\sup_{x,y \in B_e(0,R)}\frac{|X_i(x)-X_i(y)|}{|x-y|}$. Therefore there exists $C>0$ such that $c_i^{\,j}, c_i\leq C$ for any $j\in \mathbb N $ and $i=1,\dots, m$, which gives
\[
	\left|\gamma^j_{h,x}(\tau)-\gamma_{h,x}(\tau)\right|\leq \int_0^\tau\left(mC \left|\gamma^j_{h,x}(s)-\gamma_{h,x}(s)\right| +\sum_{i=1}^m\sup_{B_e(0,R)}|X_i^j-X_i|\right)ds.
\]
	We can therefore apply Gr\"onwall's Lemma (see \cite{Gronwall}) to get
\[
	\left|\gamma^j_{h,x}(t_j)-\gamma_{h,x}(t_j)\right|\leq t_j\left(\sum_{i=1}^m\sup_{B_e(0,R)}|X^j_i-X_i|\right)e^{mCt_j}\leq \frac{1}{2}.
\]
Notice that $\gamma_{h,x}(t_j)\in K'$ and $\gamma_{h,x}^j(t_j) \in \partial B_e(0,R)$: this contradicts the definition of $R$, giving $t_j= T$. The lemma is proved. 
\end{proof}

\begin{lemma}\label{lemmapointwise}
	Fix $\varepsilon \in (0,1)$ and a compact set $K$ in $\R^n$. Then, for every $T>0$ there exists $J_2=J_2(K,T,\varepsilon) \in \mathbb N$ such that for every $x\in K$, $j\geq J_2$,   $h\in L^\infty([0,T];\R^m)$ with $\|h\|\leq 1$ and $t\in [0,T]$ one has
	\[
	|\gamma_{h,x}^j(t)-\gamma_{h,x}(t)|\leq \varepsilon
	\]
\end{lemma}
\begin{proof}
Let $J_1=J_1(K,T)$ and $R=R(K,T)$ be given by Lemma \ref{welldefined} and let $C>0$ be the constant appearing in its proof. We can reason as in Lemma \ref{welldefined} above and use Gr\"onwall's Lemma to get, for any $x,j,h,t$ as in the statement, that
\[
	\left|\gamma^j_{h,x}(t)-\gamma_{h,x}(t)\right|\leq t\left(\sum_{i=1}^m\sup_{B_e(0,R)}|X^j_i-X_i|\right)e^{mCt}.
\]
The proof is then accomplished by  choosing $J_2\geq J_1$ sufficiently large to have
\[
T\left(\sum_{i=1}^m\sup_{B_e(0,R)}|X^j_i-X_i|\right)e^{mCT}<\varepsilon.
\]
\end{proof}

Clearly, $J_2$ can be chosen with the additional property that   $J_2(K,T_1,\varepsilon)\leq J_2(K,T_2,\varepsilon)$ whenever $0<T_1\leq T_2$.

\begin{theorem}\label{teo:convunifd_j}
Let $X=(X_1,\dots,X_m)$ and $X^j=(X_1^j,\dots,X_m^j)$, $j\in \mathbb{N}$, be $m$-tuples of linearly independent smooth vector fields on $\mathbb R^n$ such that  $X$ satisfies the Chow-H\"ormander condition and its CC balls are bounded in $\R^n$; assume that, for every $i=1,\dots,m$, $X^j_i\to X_i$ in $C^\infty_{loc}(\R^n)$ as $j\to\infty$. Then the sequence  $(d_j)_j$ converges to $d$ in $L^\infty_{loc}(\mathbb{R}^n\times \mathbb{R}^n)$ as $j\to +\infty $. 
\end{theorem}
\begin{proof}
Let $K\subset\R^n$ be a fixed compact set.

	We first prove that for every $\varepsilon \in (0,1)$ there exists $J_3=J_3(K,\varepsilon)\in \mathbb N$  such that for every $x,y\in K$ and $j\geq J_3$ one has
	\[
	d_j(x,y)\leq d(x,y)+\varepsilon.
	\]
Consider $x,y\in K$; by the existence of geodesics, there exists $h\in L^\infty([0,1];\mathbb{R}^m)$ such that $\|h\|_{L^\infty}= d(x,y)$ and $\gamma_{h,x}(1)=y$. We set $y_j:=\gamma_{h,x}^j(1)$ and consider $J_0$ and $C_K>0$ as given by Theorem \ref{uniformity}. For $j\geq J_3:=\max\{J_0,J_2(K,\mathrm{diam}_d K,(\varepsilon/C_K)^{s(K)})\}$ we have
	\[
	|y_j-y|=|\gamma_{h,x}^j(1)-\gamma_{h,x}(1)|\leq \left(\frac{\varepsilon}{C_K}\right)^{s(K)}.
	\]
By Theorem \ref{uniformity} we deduce that $d_j(y_j,y)\leq \varepsilon$; in particular, for any $j\geq J_3$ one has
	\begin{equation}\label{limsup1}
	d_j(x,y)\leq d_j(x,y_j)+d_j(y_j,y)\leq d(x,y)+\varepsilon,
	\end{equation}
as claimed. Notice also that $ \sup_{j\geq J_3}\text{diam}_{d_j} K\leq \mathrm{diam}_d K + 1=:L$ is finite.
	
		We now prove that for any $x, y\in K$ and $\varepsilon \in (0,1)$ there exists $J_4=J_4(K,x,y,\varepsilon)\in \mathbb N$ such that for every $j\geq J_4$ 
	\begin{equation}\label{eq:canestro}
		d(x,y)\leq d_j(x,y)+\varepsilon.
	\end{equation}
For every $j\geq J_3$ let $h^j\in L^\infty([0,1];\mathbb{R}^m)$ be such that
	\[
	\gamma_{h^j,x}^j(1)=y\qquad\text{and}\qquad 
	\|h^j\|_{L^\infty}= d_j(x,y)\leq L.
	\]
The sequence $(h^j)_j$ is bounded in $L^\infty$ and therefore there exists a subsequence $(h^{j_\ell})_\ell$ and $h\in L^\infty([0,1];\R^m)$ such that 
	\[
	h^{j_\ell}\stackrel{*}{\rightharpoonup}h\text{ in }L^\infty\qquad\text{and}\qquad
	\lim_{\ell\to\infty} \|h^{j_\ell}\|_{L^\infty}=\liminf_{j\to\infty}\|h^{j}\|_{L^\infty}=\liminf_{j\to\infty} d_j(x,y).
	\]
	Denoting $\gamma^{j_\ell}:=\gamma_{h^{j_\ell},x}^{j_\ell}$ and considering $R=R(K,L)>0$ as given by Lemma \ref{welldefined}, one has $\gamma^{j_\ell}([0,1])\subseteq B_e(0,R)$. Since $X^j_i$ are converging to $X_i$ uniformly in $C^\infty$ ($i=1,\dots, m$), such vector fields are equibounded on $B_e(0,R)$. By Ascoli-Arzel\`a Theorem, up to a further subsequence, there exists a curve $\gamma\in AC([0,1],\R^n)$ such that $\gamma^{j_\ell}$ uniformly converges to $\gamma$ in $[0,1]$ as $\ell\to \infty $. For every $t\in [0,1]$ one has 
		\[
		\gamma^{j_\ell}(t)=x+\int_0^t\sum_{i=1}^m h_i^{j_\ell}(s) X_i^{j_\ell}(\gamma^ {j_\ell}(s))ds
		\]
and, taking into account that $X_i^{j_\ell}\circ\gamma^{j_\ell}\to X_i\circ \gamma$ uniformly in $[0,1]$ and that $h^j\stackrel{*}{\rightharpoonup}h$ in $L^\infty$, by letting $\ell\to\infty$ one gets
	\[
	\gamma(t)=x+\int_0^t\sum_{i=1}^m h_i(s)X_i(\gamma(s))ds.
	\]
In particular $\gamma=\gamma_{h,x}$, $\gamma(1)=y$ and 
\[
d(x,y)\leq \|h\|_{L^\infty}\leq \liminf_{\ell\to\infty} \|h_{j_\ell}\|_{L^\infty}=\liminf_{j\to\infty} d_j(x,y),
\]
which proves \eqref{eq:canestro}.

By the compactness of $K$ we can find $x_1,\dots, x_k\in K$ such that $K\subseteq \bigcup_{\ell=1}^k B(x_\ell, \varepsilon)$. Using Theorem \ref{uniformity} and \eqref{eq:canestro} we can find $\widetilde C=\widetilde C(K)>0$ and $J_5=J_5(K,\varepsilon)\in \mathbb N $ such that for  $j\geq J_5$
	\[
	\begin{aligned}
		&B(x_\ell, \varepsilon)\subseteq  B^j(x_\ell, \widetilde C\varepsilon^{1/{s(K)}})\qquad&&\forall\;\ell=1,\dots,k\\
		& d(x_{\ell_1},x_{\ell_2})\leq d_j(x_{\ell_1},x_{\ell_2}) +\varepsilon\qquad&&\forall\;\ell_1,\ell_2=1,\dots, k.
	\end{aligned}
	\]
For  every $x,y \in K$ we can find $x_{\ell_1}, x_{\ell_2}\in K$ (with $1\leq \ell_1, \ell_2\leq k$) such that $x\in B(x_{\ell_1},\varepsilon)$ and $y\in B(x_{\ell_2},\varepsilon)$, hence for  $j\geq J_5$ we have
	\[
	\begin{aligned}
	d(x,y)&\leq d(x, x_{\ell_1})+d(x_{\ell_1},x_{\ell_2})+d(y,x_{\ell_2})\\
	& \leq \varepsilon +d_j(x_{\ell_1},x_{\ell_2})+\varepsilon +\varepsilon \\ 
	& \leq d_j(x_{\ell_1},x)+d_j(x,y)+d_j(y, x_{\ell_2})+3\varepsilon \\
	& =d_j(x,y)+3\varepsilon+2\widetilde C\varepsilon^{1/{s(K)}},
	\end{aligned}
	\]
	which, combined with \eqref{limsup1}, concludes the proof.
\end{proof}

Let us recall that, given a CC space $(\mathbb R^n, X)$, a function  $u\in L^1_{loc}(\Omega)$ is said to have {\em locally bounded $X$-variation}  if the distributional derivatives $X_1u,\dots, X_mu$ are represented by Radon measures. See e.g. \cite{capdangar,FSSC}. We denote by $BV_{X,loc}(\R^n)$ the set of functions of locally bounded $X$-variation in $\R^n$ and by $|D_Xu|$ the total variation of the vector-valued measure $D_Xu:=(X_1u,\dots, X_mu)$.

Sobolev- and Poincar\'e-type inequalities in CC spaces have been largely investigated; among the vast literature we mention only \cite{Jerison,GN,hajkos}. The following result is an easy consequence of \cite[Theorem 7.2]{BBP} or \cite[Theorem 1.1]{MM}. Notice that the latter results are proved  only when $u$ is a smooth function on $\R^n$; in order to prove Theorem \ref{teo:poincare} as  stated here one has to approximate  functions in $BV_{X,loc}$ by smooth ones (see \cite{FSSC,GN}).
	
\begin{theorem}\label{teo:poincare}
Let $X=(X_1,\dots,X_m)$ and $X^j=(X_1^j,\dots,X_m^j)$, $j\in \mathbb{N}$, be $m$-tuples of linearly independent smooth vector fields on $\mathbb R^n$ such that  $X$ satisfies the Chow-H\"ormander condition and its CC balls are bounded in $\R^n$; assume that, for every $i=1,\dots,m$, $X^j_i\to X_i$ in $C^\infty_{loc}(\R^n)$ as $j\to\infty$.  Then, for every compact set $K\subseteq \mathbb{R}^n$ there exist $C_P>1$, $\alpha\geq 1,\ R_P>0$ and $J\in\N$ such that for every $j\geq J$, $u\in BV_{X^j,loc}(\R^n)$,  $x\in K$ and $r\in (0,R_P)$ one has
	\begin{equation}\label{eq:poincare}
	\int_{B^j} \left|u-u(B^j)\right|d\mathscr{L}^n\leq C_P\,r\:|D_{X^j}u|(\alpha B^j),
	\end{equation}
where $B^j:=B^j(x,r)$ and $ u(B^j)=\fint_{B^j} u\,d\mathscr{L}^n$.
\end{theorem}

We can then state our main application. See \cite[Section 8]{hajkos} for more references about compactness results for Sobolev or BV functions in CC spaces.

\begin{theorem}\label{teo:applicazione}
Let $X=(X_1,\dots,X_m)$ and $X^j=(X_1^j,\dots,X_m^j)$, $j\in \mathbb{N}$, be $m$-tuples of linearly independent smooth vector fields on $\mathbb R^n$ such that  $X$ satisfies the Chow-H\"ormander condition and its CC balls are bounded in $\R^n$; assume that, for every $i=1,\dots,m$, $X^j_i\to X_i$ in $C^\infty_{loc}(\R^n)$ as $j\to\infty$.  Let $u_j\in BV_{X^j,loc}(\R^n)$ be a sequence of functions that is locally uniformly  bounded in $BV$, i.e., such that for any compact set $K\subset\R^n$ there exists $M>0$ such that
\[
\forall j\in\N\qquad \|u_j\|_{L^1(K)} + |D_{X^j}u_j|(K)\leq M<\infty.
\]
Then, there exist $u\in BV_{X,loc}(\R^n)$ and a subsequence $(u_{j_h})_h$ of $(u_j)_j$ such that $u_{j_h}\to u$ in $L^1_{loc}(\R^n)$ as $h\to\infty$. Moreover, for any bounded open set $\Omega\subset\R^n$ the semicontinuity of the total variation
\[
|D_Xu|(\Omega) \leq \liminf_{j\to\infty} |D_{X^j} u_j|(\Omega)
\]
holds.
\end{theorem}
\begin{proof}
We use Theorem \ref{structureconvergence} with $X=\R^n$, $\lambda=\mathscr{L}^n$, $\delta=q=1$, $\mu_j:=|D_{X^j}u|$ and $d, d_j$ the CC distances associated with $X,X^j$ respectively. Assumption (i) follows from Theorem \ref{teo:convunifd_j}, while the local doubling property (ii) of $d$ is a well-known fact (see e.g. \cite{NSW}). The validity of (iii) (with $\delta=q=1$) follows from Theorem \ref{teo:poincare}, while (iv) is satisfied by assumption.

Theorem \ref{structureconvergence} ensures that, up to subsequences, $u_j$ converges to some $u$ in $L^1_{loc}(\R^n)$; we need to show that $u\in BV_{X,loc}(\R^n)$. To this aim, for any $i=1,\dots,m$ we denote by $X^*_i$ the formal adjoint to $X_i$ and write
\[
X_i(x)=\sum_{k=1}^n a_{i,k}(x)\partial_k\qquad\text{and}\qquad X_i^j(x)=\sum_{k=1}^n a_{i,k}^j(x)\partial_k
\]
for suitable smooth functions $a_{i,k},a_{i,k}^j$. Then, for any bounded open set $\Omega\subset\R^n $ and any test function $\varphi\in C^1_c(\Omega)$ we have
\[
\begin{split}
 \int_\Omega u \,X_i^*\varphi\:d\mathscr L^n =&  \int_\Omega u\sum_{k=1}^n\partial_k(a_{i,k}\varphi)\:d\mathscr L^n
\ = \ \lim_{j\to\infty} \int_\Omega u_j\sum_{k=1}^n\partial_k(a_{i,k}^j\varphi)\:d\mathscr L^n\\
 = & -\lim_{j\to\infty} \int_\Omega \varphi\,dX_i^ju_j
\ \leq \ \|\varphi\|_{L^\infty(\Omega)} \liminf_{j\to\infty} |D_X^j u_j|(\Omega) <\infty.
\end{split}
\]
This proves that $u\in BV_{X,loc}(\R^n)$ as well as the semicontinuity of the total variation. The proof is accomplished.
\end{proof}

\begin{rmk}{\rm
We conjecture that, when the CC space $(\R^n,X)$ is {\em equiregular}, the convergence $u_{j_h}\to u$ in Theorem \ref{teo:applicazione} holds in $L^q_{loc}$ for any $q\in[1,\tfrac Q{Q-1}[$, where $Q$ is the Hausdorff dimension of $(\R^n,X)$. This would easily follow in case the Poincar\'e inequality \eqref{eq:poincare} could be strengthened to 
\[
\|u-u(B^j)\|_{L^q(B^j)} \leq C_P\, r^\delta \,|D_{X^j}u|(\alpha B^j)
\]
for some $\delta>0$ (arguably, $\delta=\tfrac Qq+1-Q$). The key point would be proving that the constant $C_P$ can be chosen independent of $j$ but, as far as we know, no investigation in this direction  has been attempted in the literature.
}\end{rmk}

\begin{rmk}{\rm
Theorems \ref{teo:convunifd_j},  \ref{teo:poincare} and \ref{teo:applicazione}  hold also under a slightly weaker assumption: it is indeed enough that, for any compact set $K\subset\R^n$, the convergence  $X_i^j\to X_i$ holds in $C^k(K)$ for a suitable $k=k(K)$ (actually, $k$ depends only on $s(K)$) that one could  explicitly compute. See \cite{BBP,MM} for more details.
}\end{rmk}

\end{document}